\documentclass[12pt]{article}
\usepackage{amssymb,amsfonts,amsmath,amsrefs,amsthm,hyperref,bbm,dsfont,tikz,wrapfig,tikz-cd}

\usepackage{vmargin}
\setpapersize{A4}
\setmargnohfrb{3cm}{3cm}{3cm}{3cm}

\newcommand{\8}{\infty}

\newcommand{\Z}{\mathbb{Z}}

\newcommand{\R}{\mathbb{R}}
\newcommand{\C}{\mathbb{C}}
\newcommand{\N}{\mathbb{N}}

\newcommand{\T}{\mathbb{T}}
\newcommand{\D}{\mathbb{D}}

\newcommand{\dil}{\mathrm{dil}}

\newcommand{\spa}{\mathrm{span}}
\newcommand{\Ker}{\mathrm{Ker~}}

\newcommand{\conv}{\mathrm{conv}}

\newcommand{\Co}{\mathcal{C}}

\newcommand{\Lo}{\mathcal{L}}

\newcounter{erz}[section] \numberwithin{erz}{section}
\newtheorem{theorem}[erz]{Theorem}
\newtheorem{lemma}[erz]{Lemma}
\newtheorem{proposition}[erz]{Proposition}

\newtheorem{corollary}[erz]{Corollary}

\theoremstyle{remark}
\newtheorem{remark}[erz]{Remark}
\newtheorem{example}[erz]{Example}

\begin{document}
\title{Which multiplication operators are surjective isometries?}
\author{Eugene Bilokopytov\footnote{Email address bilokopi@myumanitoba.ca, erz888@gmail.com.}}
\maketitle

\begin{abstract}
Let $\mathbf{F}$ be a Banach space of continuous functions over a connected locally compact space $X$. We present several sufficient conditions on $\mathbf{F}$ guaranteeing that the only multiplication operators on $\mathbf{F}$ that are surjective isometries are scalar multiples of the identity. The conditions are given via the properties of the inclusion operator from $\mathbf{F}$ into $\Co\left(X\right)$, as well as in terms of geometry of $\mathbf{F}$. An important tool in our investigation is the notion of Birkhoff Orthogonality.


MSC2010 46B20, 46E15, 47B38
\end{abstract}

\section{Introduction}

Normed spaces of functions are ubiquitous in mathematics, especially in analysis. These spaces can be of a various nature and exhibit different types of behavior, and in this work we discuss some questions related to these spaces from a general, axiomatic viewpoint. The class of linear operators that capture the very nature of of the spaces of functions is the class of weighted composition operators (WCO). Indeed, the operations of multiplication and composition can be performed on any collection of functions, while there are several Banach-Stone-type theorems which show that the WCO's are the only operators that preserve various kinds of structure (see e.g. \cite{fj} and \cite{gj} for more details).

In this article we continue our investigation (see \cite{erz}) of the general framework which allows to consider any Banach space that consists of continuous (scalar-valued) functions, such that the point evaluations are continuous linear functionals, and of WCO's on these spaces.\medskip

First, let us define precisely what we mean by a normed space of continuous functions. Let $X$ be a topological space (a \emph{phase space}) and let $\Co\left(X\right)$ denote the space of all continuous complex-valued functions over $X$ endowed with the compact-open topology. A \emph{normed space of continuous functions} (NSCF) over $X$ is a linear subspace $\mathbf{F}\subset\Co\left(X\right)$ equipped with a norm that induces a topology, which is stronger than the compact-open topology, i.e. the inclusion operator $J_{\mathbf{F}}:\mathbf{F}\to\Co\left(X\right)$ is continuous, or equivalently the unit ball $B_{\mathbf{F}}$ is bounded in $\Co\left(X\right)$. If $\mathbf{F}$ is a linear subspace of $\Co\left(X\right)$, then the \emph{point evaluation} at $x\in X$ on $\mathbf{F}$ is the linear functional $x_{\mathbf{F}}:\mathbf{F}\to\C$, defined by $x_{\mathbf{F}}\left(f\right)=f\left(x\right)$. If $\mathbf{F}$ is a NSCF, then all point evaluations are bounded on $\mathbf{F}$. Conversely, if $\mathbf{F}\subset\Co\left(X\right)$ is equipped with a complete norm such that $x_{\mathbf{F}}\in \mathbf{F}^{*}$, for every $x\in X$, then $\mathbf{F}$ is a NSCF. We will call a NSCF $\mathbf{F}$ over $X$ \emph{(weakly) compactly embedded} if $J_{\mathbf{F}}$ is a (weakly) compact operator, or equivalently, if $B_{\mathbf{F}}$ is (weakly) relatively compact in $\Co\left(X\right)$.\medskip

Let $X$ and $Y$ be topological spaces, and let $\Phi:Y\to X$ and $\omega:Y\to\C $ (not necessarily continuous). A \emph{weighted composition operator} (WCO) with \emph{composition symbol} $\Phi$ and \emph{multiplicative symbol} $\omega$ is a linear map $W_{\Phi,\omega}$ from the space of all complex-valued functions on $X$ into the analogous space over $Y$ defined by $$\left[W_{\Phi,\omega}f\right]\left(y\right)=\omega\left(y\right)f\left(\Phi\left(y\right)\right),$$ for $y\in Y$. Let $\mathbf{F}\subset\Co\left(X\right)$, $\mathbf{E}\subset\Co\left(Y\right)$ be linear subspaces. If $W_{\Phi,\omega}\mathbf{F}\subset\mathbf{E} $, then we say that $W_{\Phi,\omega}$ is a weighted composition operator from $\mathbf{F}$ into $\mathbf{E}$ (we use the same notation $W_{\Phi,\omega}$ for what is in fact $W_{\Phi,\omega}\left|_{\mathbf{F}}\right.$). In particular, if $X=Y$, we will denote $M_{\omega}=W_{Id_{X},\omega}$ \footnote{If $X$ is a set, then by $Id_{X}$ we denote the identity map on $X$.}, and call it the \emph{multiplication operator} (MO) with symbol (or \emph{weight}) $\omega$. If in this case $\mathbf{F}=\mathbf{E}$, then we will call $\omega$ a \emph{multiplier} of $\mathbf{F}$. If $\mathbf{F}$ and $\mathbf{E}$ are both complete NSCF's, then any WCO between these spaces is automatically continuous due to Closed Graph theorem. However, in concrete cases it can be very difficult to determine all WCO's between a given pair of NSCF's. In particular, it is difficult to determine all multipliers of a NSCF (see e.g. \cite{ms} and \cite{vukotic}, where the multiplier algebras of some specific families of NSCF's are described).\medskip

WCO's may be viewed as morphisms in the category of NSCF's. In the light of this fact it is important to be able to characterize WCO's with some specific properties. In this article we focus on one such property -- being a \emph{unitary}, i.e. a surjective isometry, or an isometric isomorphism. More specifically, we consider the following rigidity property of a NSCF $\mathbf{F}$ over $X$: if $\mathbf{E}$ is a NSCF over $Y$, $\omega,\upsilon:Y\to\C$ and $\Phi:Y\to X$ are such that both $W_{\Phi,\omega}$ and $W_{\Phi,\upsilon}$ are unitaries from $\mathbf{F}$ into $\mathbf{E}$, then there is $\lambda\in\C$, $\left|\lambda\right|=1$ with $\upsilon=\lambda\omega$. In particular, we are looking for conditions on $\mathbf{F}$ such that the only unitary MO's on $\mathbf{F}$ are the scalar multiples of the identity.

Some related problems were studied (see e.g. \cite{admv}, \cite{ac1}, \cite{ac2}, \cite{bn}, \cite{le1}, \cite{le2}, \cite{matache}, \cite{zhao}, \cite{nz1} and \cite{nz2}). Note that in these articles the class of operators under consideration is wider (e.g. unitary WCO's, or isometric MO's, as opposed to unitary MO's), but these operators are considered on the narrower classes of NSCF's.\medskip

Let us describe the contents of the article. In Section 2 we gather some elementary properties of NSCF's and WCO's. In particular, we characterize weakly compactly embedded NSCF's (Theorem \ref{barrell2}) and prove that a WCO between complete NSCF's with a surjective composition symbol is a linear homeomorphism if and only if its adjoint is bounded from below (part (iii) of Corollary \ref{winj2}). Section 3 is dedicated to the main problem of the article, and in particular it contains the main results (Theorem \ref{main} and Proposition \ref{mainw}), which give sufficient conditions for a NSCF to have the rigidity properties described above. In Section 4 we consider an interpretation of Theorem \ref{main} for abstract normed spaces, as opposed to NSCF's. Also, we study some properties of Birkhoff (-James) orthogonality which is an important tool in our investigation. Finally, we consider the class of nearly strictly convex normed spaces that includes strictly convex and finitely dimensional normed spaces, and arises naturally when studying NSCF's in the context of Birkhoff orthogonality.\medskip

\textbf{Some notations and conventions.} Let $\D$ (or $\overline{\D}$) be the open (or closed) unit disk on the plane $\C$, and let $\T=\partial \D$ be the unit circle. For a linear space $E$ let $E'$ be the algebraic dual of $E$, i.e. the linear space of all linear functionals on $E$.

\section{Preliminaries}

In this section we discus some basic properties of NSCF's and WCO's. Let us start with NSCF's. We will often need to put certain restrictions on the phase spaces of NSCF's. A Hausdorff topological space $X$ is called \emph{compactly generated}, or a \emph{k-space} whenever each set which has closed intersections with all compact subsets of $X$ is closed itself. It is easy to see that all first countable (including metrizable) and all locally compact Hausdorff spaces are compactly generated. Moreover, Arzela-Ascoli theorem describes the compact subsets of $\Co\left(X\right)$ in the event when $X$ is compactly generated, which further justifies the importance of this class of topological spaces. Details concerning the mentioned facts and some additional information about the compactly generated spaces can be found in \cite[3.3]{engelking}.\medskip

Let us characterize (weakly) compactly embedded NSCF's using the following variation of a classic result (see \cite{bartle}, \cite[VI.7, Theorem 1]{ds}, \cite[3.7, Theorem 5]{grot}, \cite{wada}).

\begin{theorem}\label{barrell} Let $\mathbf{F}$ be a NSCF over a Hausdorff space $X$. Then $\kappa_{\mathbf{F}}$ is a weak* continuous map from $X$ into $\mathbf{F}^{*}$. Moreover, the following equivalences hold:
\item[(i)] $\mathbf{F}$ is weakly compactly embedded if and only if $\kappa_{\mathbf{F}}$ is weakly continuous.
\item[(ii)] If $\kappa_{\mathbf{F}}$ is norm-continuous, then $\mathbf{F}$ is compactly embedded. The converses holds whenever $X$ is compactly generated.
\end{theorem}

More generally, every linear map $T$ from a linear space $F$ into $\Co\left(X\right)$ generates a weak* continuous map $\kappa_{T}:X\to F'$ defined by  $\left<f,\kappa_{T}\left(x\right)\right>=\left[Tf\right]\left(x\right)$, for $x\in X$ and $f\in F$. In this case $\kappa_{T}\left(A\right)^{\bot}=\kappa_{T}\left(\overline{A}\right)^{\bot}$, for any $A\subset X$, and $\Ker T=\kappa_{T}\left(X\right)^{\bot}$.\medskip

\begin{remark}
Clearly, every compactly embedded NSCF's is weakly compactly embedded. On the other hand, it follows from the theorem above that any reflexive NSCF is also weakly compactly embedded.

If $X$ is a domain in $\C^{n}$, i.e. an open connected set, and $\mathbf{F}$ is a NSCF over $X$ that consists of holomorphic functions, then $\mathbf{F}$ is compactly embedded. Indeed, by Montel's theorem (see \cite[Theorem 1.4.31]{scheidemann}), $B_{\mathbf{F}}$ is relatively compact in $\Co\left(X\right)$, since it is a bounded set that consists of holomorphic functions.
\qed\end{remark}

For a NSCF $\mathbf{F}$ over $X$ let $\overline{B_{\mathbf{F}}}^{\Co\left(X\right)}$ be the closure of $B_{\mathbf{F}}$ in $\Co\left(X\right)$. Since $\overline{B_{\mathbf{F}}}^{\Co\left(X\right)}$ is bounded, closed, convex and balanced, we can generate a NSCF with the closed unit ball $\overline{B_{\mathbf{F}}}^{\Co\left(X\right)}$. Namely, define $\widehat{\mathbf{F}}=\left\{\alpha f\left|\alpha>0,~ f\in \overline{B_{\mathbf{F}}}^{\Co\left(X\right)}\right.\right\}$, which is a linear subspace of $\Co\left(X\right)$, and endow it with the norm being the Minkowski functional of $\overline{B_{\mathbf{F}}}^{\Co\left(X\right)}$. Since $\overline{B}_{\widehat{\mathbf{F}}}=\overline{B_{\mathbf{F}}}^{\Co\left(X\right)}$ is bounded in $\Co\left(X\right)$, it follows that $\widehat{\mathbf{F}}$ is a NSCF over $X$. It is clear that $\mathbf{F}$ is (weakly) compactly embedded if and only if $\widehat{\mathbf{F}}$ is (weakly) compactly embedded. It turns out, that the fact that $\mathbf{F}$ is weakly compactly embedded can be further characterized in terms of $\widehat{\mathbf{F}}$ and $\overline{B_{\mathbf{F}}}^{\Co\left(X\right)}$.

\begin{theorem}\label{barrell2}
Let $\mathbf{F}$ be a NSCF over a Hausdorff space $X$. Then the following are equivalent:
\item[(i)] $\mathbf{F}$ is weakly compactly embedded;
\item[(ii)] $\overline{B_{\mathbf{F}}}^{\Co\left(X\right)}$ is compact with respect to the pointwise topology on $\Co\left(X\right)$;
\item[(iii)] $\widehat{\mathbf{F}}=\left(\spa~\kappa_{\mathbf{F}}\left(X\right)\right)^{*}$ (as normed spaces) via the bilinear form induced by $\left<x_{\mathbf{F}},f\right>=f\left(x\right)$.
\end{theorem}
\begin{proof}
(iii)$\Rightarrow$(ii): If (iii) holds, then the pointwise topology on $\widehat{\mathbf{F}}$ coincides with the weak* topology. Hence, the unit ball $\overline{B_{\mathbf{F}}}^{\Co\left(X\right)}$ is pointwise compact due to Banach-Alaoglu theorem.

(ii)$\Leftrightarrow$(i): From the definition of a NSCF, $B_{\mathbf{F}}$ is bounded in $\Co\left(X\right)$. Hence, this set is weakly relatively compact if and only if it is relatively compact with respect to the pointwise topology on $\Co\left(X\right)$ (see \cite[4.3, Corollary 2]{floret}).

(i)$\Rightarrow$(iii): If $\mathbf{F}$ is weakly compactly embedded then $J_{\mathbf{F}}$ is weakly compact, and so $J_{\mathbf{F}}^{**}$ maps $\mathbf{F}^{**}$ into $\Co\left(X\right)$ with $J_{\mathbf{F}}^{**}\overline{B}_{\mathbf{F}^{**}}=\overline{B_{\mathbf{F}}}^{\Co\left(X\right)}=\overline{B}_{\widehat{\mathbf{F}}}$ (the proof of \cite[VI.4, Theorem 2]{ds} carries over to the case when the target space is locally convex, see also \cite[2.18, Theorem 13]{grot}). Consequently, $J_{\mathbf{F}}^{**}B_{\mathbf{F}^{**}}=B_{\widehat{\mathbf{F}}}$. Indeed, if $f\in B_{\widehat{\mathbf{F}}}$, there is $g\in \overline{B}_{\mathbf{F}^{**}}$ such that $J_{\mathbf{F}}^{**} g=\frac{f}{\|f\|_{\widehat{\mathbf{F}}}}$. Then $J_{\mathbf{F}}^{**} \left(\|f\|_{\widehat{\mathbf{F}}}g\right)=f$, and since $\|g\|\le 1$ and $\|f\|_{\widehat{\mathbf{F}}}<1$ it follows that $f\in J_{\mathbf{F}}^{**} B_{\mathbf{F}^{**}}$. On the other hand, as $\|J_{\mathbf{F}}^{**}\|_{\Lo\left(\mathbf{F}^{**}, \widehat{\mathbf{F}}\right)}\le 1$, it follows that $J_{\mathbf{F}}^{**}g\in B_{\widehat{\mathbf{F}}}$, for any $g\in B_{\mathbf{F}^{**}}$.

Hence, $J_{\mathbf{F}}^{**}$ is a quotient map from $\mathbf{F}^{**}$ onto $\widehat{\mathbf{F}}$ (see the proof of \cite[Lemma 2.2.4]{istr}), and so $\widehat{\mathbf{F}}\simeq \mathbf{F}^{**}\slash \Ker J_{\mathbf{F}}^{**}$. For $g\in \mathbf{F}^{**}$ we have that $g\in \Ker J_{\mathbf{F}}^{**}$ if and only if $\left[J_{\mathbf{F}}^{**}g\right]\left(x\right)=0$, for every $x\in X$. By definition, $\left[J_{\mathbf{F}}^{**}g\right]\left(x\right)=\left<g,x_{\mathbf{F}}\right>$, and so $\Ker J_{\mathbf{F}}^{**}=\kappa_{\mathbf{F}}\left(X\right)^{\bot}$ in $\mathbf{F}^{**}$. Finally, since $\mathbf{F}^{**}\slash \kappa_{\mathbf{F}}\left(X\right)^{\bot}$ is isometrically isomorphic to $\left(\spa~\kappa_{\mathbf{F}}\left(X\right)\right)^{*}$ (see the proof of \cite[Proposition 2.6]{fhhmz}), the result follows.
\end{proof}

\begin{corollary}\label{barrell3}
Let $\mathbf{F}$ be a NSCF over a Hausdorff space $X$. Then\linebreak $\mathbf{F}=\left(\spa~\kappa_{\mathbf{F}}\left(X\right)\right)^{*}$ (as normed spaces) if and only if $\mathbf{F}$ is weakly compactly embedded and $\overline{B}_{\mathbf{F}}$ is closed in $\Co\left(X\right)$.
\end{corollary}

Let us consider some examples of NSCF's.

\begin{example}
Let $\Co_{\8}\left(X\right)$ be the space of all bounded continuous functions on $X$, with the supremum norm $\|f\|=\sup\limits_{x\in X}\left|f\left(x\right)\right|$. It is easy to see that $\Co_{\8}\left(X\right)$ is a complete NSCF, but if $X$ is not a discrete topological space, then $\Co_{\8}\left(X\right)$ is NOT weakly compactly embedded. Indeed, its closed unit ball  $\Co\left(X,\overline{\D}\right)$ is not a pointwise compact set since any $f:X\to \overline{\D}$ can be approximated by elements of $\Co\left(X,\overline{\D}\right)$ in the pointwise topology.\medskip
\qed\end{example}

\begin{example}\label{lip}
Let $\left(X,d\right)$ be a metric space and let $z\in X$. For $f:X\to\C$ define \linebreak$\dil f= \sup\left\{\frac{\left|f\left(x\right)-f\left(y\right)\right|}{d\left(x,y\right)}\left|x,y\in X,~ x\ne y\right.\right\}$. This functional generates a NSCF \linebreak $Lip\left(X,d\right)=\left\{f:X\to\C\left|\dil f<+\8\right.\right\}$ with the norm $\|f\|=\dil f+\left|f\left(z\right)\right|$. One can show that $\mathbf{F}=Lip\left(X,d\right)$ is a complete NSCF with $\|x_{\mathbf{F}}\|=\max\left\{1,d\left(x,z\right)\right\}$ and $\|x_{\mathbf{F}}-y_{\mathbf{F}}\|=d\left(x,y\right)$, for every $x,y\in X$ (the proof is a slight modification of the proof from \cite{ae}). Hence, $Lip\left(X,d\right)$ is compactly embedded due to part (ii) of Theorem \ref{barrell}. Moreover, it not difficult to show that $\overline{B}_{\mathbf{F}}$ is closed in $\Co\left(X\right)$, and so $\mathbf{F}=\left(\spa~\kappa_{\mathbf{F}}\left(X\right)\right)^{*}$, due to Corollary \ref{barrell3}.
\qed\end{example}

Let us now consider basic properties of WCO's and in particular MO's. We start with a well-known fact (see e.g. \cite[Proposition 2.4 and Corollary 2.5]{erz}).

\begin{proposition}\label{rec}
Let $X$ and $Y$ be topological spaces. Let $\mathbf{F}\subset\Co\left(X\right)$ and $\mathbf{E}\subset\Co\left(Y\right)$ be linear subspaces, an let $T$ be a linear map from $\mathbf{F}$ into $\mathbf{E}$. Then $T=W_{\Phi,\omega}$, for $\Phi:Y\to X$ and $\omega:Y\to\C $ if and only if $T'\kappa_{\mathbf{E}}\left(y\right)=\omega\left(y\right)\kappa_{\mathbf{F}}\left(\Phi\left(y\right)\right)$, for every $y\in Y$. In other words, $T$ is a WCO if and only if $T'\kappa_{\mathbf{E}}\left(Y\right)\subset\C\kappa_{\mathbf{F}}\left(X\right)$.
\end{proposition}

In particular, if $X=Y$ and $\mathbf{F}=\mathbf{E}$, then $T$ is a MO if and only if $x_{\mathbf{F}}$ is an eigenvector of $T$ (or else $x_{\mathbf{F}}=0_{\mathbf{F}'}$), for every $x\in X$. Then the multiplier is the correspondence between $x$ and the eigenvalue of $T$ for $x_{\mathbf{F}}$. Also, it follows that $$\Ker W_{\Phi,\omega}=\left(W_{\Phi,\omega}'\kappa_{\mathbf{E}}\left(Y\right)\right)^{\bot}=\kappa_{\mathbf{F}}\left(\Phi\left(Y\backslash \omega^{-1}\left(0\right)\right)\right)^{\bot}=\kappa_{\mathbf{F}}\left(\overline{\Phi\left(Y\backslash \omega^{-1}\left(0\right)\right)}\right)^{\bot}.$$

Note that in general we cannot reconstruct the symbols of a WCO from its data as a linear operator between certain NSCF's, in the sense that the equality of WCO's does not imply the equality of their symbols.

\begin{example}\label{obstructions}Let $\mathbf{F}$ and $\mathbf{E}$ be NSCF's over topological spaces $X$ and $Y$ respectively.
\begin{itemize}
\item If $x\in X$ is such that $x_{\mathbf{F}}=0_{\mathbf{F}^{*}}$, i.e. $f\left(x\right)=0$, for every $f\in \mathbf{F}$, then $M_{\omega}$ on $\mathbf{F}$ does not depend on $\omega\left(x\right)$, in the sense that if $\omega,\upsilon:X\to\C$ coincide outside of $x$, then $M_{\omega}=M_{\upsilon}$.
\item If $\omega\left(y\right)=0$, for $y\in Y$, then $W_{\Phi,\omega}$ does not depend on $\Phi\left(y\right)$, for $\Phi:Y\to X$.
\item More generally, we can construct a WCO with nontrivial symbols which is equal to the identity on $\mathbf{F}$ if there are two distinct points in $X$ such that the point evaluations on $\mathbf{F}$ at these points are linearly dependent.\qed
\end{itemize}
\end{example}

Since we are interested in investigating properties of the symbols of WCO's based on their operator properties, we need to be able to reconstruct the symbols. Hence, we have to introduce the following concepts. We will call a linear subspace $\mathbf{F}$ of $\Co\left(X\right)$ $1$-\emph{independent} if $0_{\mathbf{F}'}\not\in\kappa_{\mathbf{F}}\left(X\right)$, i.e. for every $x\in X$ there is $f\in \mathbf{F}$ such that $f\left(x\right)\ne 0$. We will say that $\mathbf{F}$ is $2$-\emph{independent} if $x_{\mathbf{F}}$ and $y_{\mathbf{F}}$ are linearly independent, for every distinct $x,y\in X$. It is easy to see that this condition is equivalent to the existence of $f,g\in\mathbf{F}$ such that $f\left(x\right)=1$, $f\left(y\right)=0$, $g\left(x\right)=0$ and $g\left(y\right)=1$. Note that if $\mathbf{F}$ is $2$-independent, it is $1$-\emph{independent} and separates points of $X$, if $\mathbf{F}$ contains nonzero constant functions, it is $1$-independent, and if $\mathbf{F}$ contains nonzero constant functions and separates points, it is $2$-independent. However, the converses to these statements do not hold.

It is easy to see that MO's from a $1$-independent NSCF determine their symbols, and WCO's from a $2$-independent NSCF also determine their symbols (see \cite[Proposition 2.8]{erz}). Moreover, some properties of the symbols of WCO can indeed be recovered (see \cite[Corollary 3.3 and Proposition 4.3]{erz}).

\begin{proposition}\label{hc} Let $\mathbf{F}$ be a NSCF over a topological space $X$. Then:
\item[(i)] If $\mathbf{F}$ is $1$-independent, then its multipliers are continuous.
\item[(ii)] If $X$ is a domain in $\C^{n}$, and $\mathbf{F}$ consists of holomorphic functions, then its multipliers can be chosen to be holomorphic, in the sense that if $T:\mathbf{F}\to \mathbf{F}$ is a continuous MO, then there is a holomorphic $\omega:X\to\C$ such that $T=M_{\omega}$.
\end{proposition}

The following examples demonstrates that we cannot relax the requirement of $1$-independence in part (i).

\begin{example}
Let $\overline{\D}\subset\C$ be endowed with the usual metric. Let\linebreak $\mathbf{F}=\left\{f\in Lip\left(\overline{\D}\right)\left|f\left(0\right)=0\right.\right\}$ with the norm $\|f\|=\dil f$, $f\in \mathbf{F}$. This is a complete compactly embedded NSCF, and the set $\left\{x\in\overline{\D}\left|x_{\mathbf{F}}=0_{\mathbf{F}^{*}}\right.\right\}$ is the singleton $\left\{0\right\}$. Define $\omega:\overline{\D}\to\T$ by $\omega\left(z\right)=\frac{z}{\left|z\right|}$, when $z\ne 0$ and $\omega\left(0\right)=1$. Clearly, $\omega$ has a non-removable discontinuity at $0$. On the other hand, let us show that $M_{\omega}$ is a bounded invertible operator on $\mathbf{F}$.

Let $f\in \mathbf{F}$ and denote $g=M_{\omega}f$. First, $g\left(0\right)=0$ and $$\left|g\left(z\right)-g\left(0\right)\right|=\left|f\left(z\right)\right|=\left|f\left(z\right)-0\right|\le \left|z-0\right|\dil f,$$ for every $z\in \overline{\D}\backslash \left\{0\right\}$. Furthermore, for distinct $z,y\in \overline{D}\backslash \left\{0\right\}$ with $\left|z\right|\ge \left|y\right|$ we get
$$\left|g\left(z\right)-g\left(y\right)\right|=
\left|\frac{z}{\left|z\right|}f\left(z\right)-\frac{y}{\left|y\right|}f\left(y\right)\right|\le
 \left|\frac{z}{\left|z\right|}\right|\left|f\left(z\right)-f\left(y\right)\right|+
\left|\frac{z}{\left|z\right|}-\frac{y}{\left|y\right|}\right|\left|f\left(y\right)\right|.$$
We have $\left|\frac{z}{\left|z\right|}\right|\left|f\left(z\right)-f\left(y\right)\right|=\left|f\left(z\right)-f\left(y\right)\right|\le \left|z-y\right|\dil f$. At the same time, $\left|\frac{z}{\left|z\right|}-\frac{y}{\left|y\right|}\right|\left|f\left(y\right)\right|\le \left|\frac{z}{\left|z\right|}-\frac{y}{\left|y\right|}\right|\left|y\right|\dil f =\left|\frac{\left|y\right|z}{\left|z\right|}-y\right|\dil f$, and using $\left|z\right|\ge \left|y\right|$ it not difficult to prove that $\left|\frac{\left|y\right|z}{\left|z\right|}-y\right|\le \left|z-y\right|$. Hence, $\left|g\left(z\right)-g\left(y\right)\right|\le 2\left|z-y\right|\dil f$, and as $y$ and $z$ were chosen arbitrarily we conclude that $\dil g\le 2 \dil f$, and so $\|M_{\omega}f\|=\|g\|\le 2\|f\|$. Since $f$ was chosen arbitrarily, we get $\|M_{\omega}\|\le 2$. As $\dil\overline{f}=\dil f$, for any $f\in \mathbf{F}$, it follows that $\|M_{\overline{\omega}}\|=\|M_{\omega}\|\le 2$, and since $\overline{w}=\frac{1}{\omega}$ we obtain $\|M_{\omega}^{-1}\|=\|M_{\frac{1}{\omega}}\|=\|M_{\overline{\omega}}\|\le 2$.\medskip\qed\end{example}

\begin{example}
Let $\overline{\D}$ and $\omega$ be as in the previous example. For $\N_{0}=n\in\N\cup\left\{0\right\}$ consider $e_{n}:\overline{\D}\to\C$ defined by $e_{n}\left(z\right)=z\frac{\omega\left(z\right)^{n}}{2^{n}}$. Note that $\left\{e_{n}\right\}_{n\in\N_{0}}$ is linearly independent (allowing infinite series), and so there is a compactly embedded Hilbert NSCF $\mathbf{E}$ whose orthonormal basis is $\left\{e_{n}\right\}_{n\in\N}$. Namely, $\mathbf{E}$ is the Reproducing Kernel Hilbert space generated by the positive semi-definite kernel $K\left(z,w\right)=\sum\limits_{n\in\N_{0}}e_{n}\left(z\right)\overline{e_{n}\left(w\right)}= \frac{4z\overline{w}}{4\left|z\right|\left|w\right|-z\overline{w}}$ (see e.g. \cite{fm}). Then $\frac{1}{2}M_{\omega}$ acts as a unilateral shift (and in particular is an isometry) on $\mathbf{E}$, and so $\|M_{\omega}\|= 2$. 

Furtheremore, using $\left\{e_{n}\right\}_{n\in\Z}$, where $e_{n}:\overline{\D}\to\C$ is defined by $e_{n}\left(z\right)=z\frac{\omega\left(z\right)^{n}}{2^{\left|n\right|}}$ one can construct a compactly embedded Hilbert NSCF for which $M_{\omega}$ is an invertible operator (but not a scalar multiple of an isometry).
\medskip\qed\end{example}

Let us now derive some properties of WCO's from the properties of their symbols.

\begin{proposition}\label{contprop} Let $X$ and $Y$ be topological spaces and let $\mathbf{F}\subset\Co\left(X\right)$ be a linear subspace. Let $\Phi,\Psi:Y\to X$ be continuous and let $\omega,\upsilon:Y\to\C$ be such that $W_{\Phi,\omega}\mathbf{F}\subset\Co\left(Y\right)$ and $W_{\Psi,\upsilon}\mathbf{F}\subset\Co\left(Y\right)$. Then:
\item[(i)] If $\Phi$ has a dense image and $\omega$ vanishes on a nowhere dense set, then $W_{\Phi,\omega}$ is an injection (cf. \cite[Proposition 2.6]{erz}).
\item[(ii)] Assume that there is a linear operator $T:\mathbf{F}\to \mathbf{F}$ such that $W_{\Psi,\upsilon}=W_{\Phi,\omega}T$. If $\Phi$ is a surjection, $\omega$ vanishes on a nowhere dense set, and there is a continuous function $\eta:Y\to\C$, such that $\upsilon=\eta\omega$, then there are maps $\Theta:X\to X$ and $\theta:X\to \C$ such that $T=W_{\Theta,\theta}$. If $\mathbf{F}$ is $2$-independent, then $\Theta\circ\Phi=\Psi$ and $\theta\circ\Phi=\eta$.
\end{proposition}
\begin{proof}
Let $Z=Y\backslash \omega^{-1}\left(0\right)$, which is a dense subset of $Y$.

(i): If $\Phi$ has a dense image, then $\overline{\Phi\left(Z\right)}=\overline{\Phi\left(\overline{Z}\right)}=\overline{\Phi\left(Y\right)}=X$, and so $\Ker W_{\Phi,\omega}=\kappa_{\mathbf{F}}\left(\overline{\Phi\left(Z\right)}\right)^{\bot}=\kappa_{\mathbf{F}}\left(X\right)^{\bot}=\left\{0\right\}$, since $\kappa_{\mathbf{F}}\left(X\right)$ is separating on $\mathbf{F}$.\medskip

(ii): If $W_{\Psi,\upsilon}=W_{\Phi,\omega}T$ then $T'W'_{\Phi,\omega}=W'_{\Psi,\upsilon}$, and so $\omega\left(y\right)T'\Phi\left(y\right)_{\mathbf{F}}=\upsilon\left(y\right)\Psi\left(y\right)_{\mathbf{F}}$, for every $y\in Y$. Hence, $T'\Phi\left(y\right)_{\mathbf{F}}=\eta\left(y\right)\Psi\left(y\right)_{\mathbf{F}}$, for each $y\in Z$.\medskip

Note that both $T'\circ\kappa_{\mathbf{F}}\circ\Phi$ and $\eta\cdot\kappa_{\mathbf{F}}\circ\Psi$ are weak* continuous maps from $Y$ into $\mathbf{F}'$. Indeed, the adjoint operator is always continuous with respect to the weak* topology, while $\kappa_{\mathbf{F}}\circ\Phi$ and $\kappa_{\mathbf{F}}\circ\Psi$ are compositions of continuous maps; finally, multiplying a weak* continuous map with a continuous function is weak* continuous since the weak* topology is linear.

Hence, $T'\circ\kappa_{\mathbf{F}}\circ\Phi$ and $\eta\cdot\kappa_{\mathbf{F}}\circ\Psi$ are weak* continuous maps from $Y$ into $\mathbf{F}'$ that coincide on a dense set $Z$, and so $T'\Phi\left(y\right)_{\mathbf{F}}=\eta\left(y\right)\Psi\left(y\right)_{\mathbf{F}}$, for every $y\in Y$. As $\Phi$ is a surjection we get that $ T'\kappa_{\mathbf{F}}\left(X\right)\subset \C \kappa_{\mathbf{F}}\left(X\right)$, and so by virtue of Proposition \ref{rec}, $T$ is a WCO, i.e. $T=W_{\Theta,\theta}$, for some $\Theta:X\to X$ and $\theta:X\to \C$. Since in this case $W_{\Psi,\upsilon}=W_{\Phi,\omega}W_{\Theta,\theta}=W_{\Theta\circ\Phi,\omega\cdot\theta\circ\Phi}$, if $\mathbf{F}$ is $2$-independent, then $\Theta\circ\Phi=\Psi$ and $\theta\circ\Phi=\eta$.
\end{proof}

\begin{corollary}\label{winj2}
Let $\mathbf{F}$ be a NSCF over a topological space $X$, let $\mathbf{E}$ be a $1$-independent NSCF over a topological space $Y$, and let $\Phi:Y\to X$ and $\omega:Y\to\C$ be such that $W_{\Phi,\omega}\in\Lo\left(\mathbf{F},\mathbf{E}\right)$. Then:
\item[(i)] If $W_{\Phi,\omega}^{*}$ is an injection, then $\omega$ does not vanish.
\item[(ii)] If $\Phi$ has a dense image and $W_{\Phi,\omega}^{*}$ is an injection, then $W_{\Phi,\omega}$ is an injection.
\item[(iii)] If $\mathbf{F}$ and $\mathbf{E}$ are Banach spaces and $\Phi$ has a dense image, then $W_{\Phi,\omega}^{*}$ is bounded from below (isometry) if and only if $W_{\Phi,\omega}$ is an linear homeomorphism (unitary).
\end{corollary}
\begin{proof}
(i),(ii): If $W_{\Phi,\omega}^{*}$ is an injection, then $\mathbf{H}=W_{\Phi,\omega}\mathbf{F}$ is dense in $\mathbf{E}$. One can show that a dense subspace of a $1$-independent NSCF is $1$-independent. Hence, if $\omega\left(y\right)=0$, then $y_{\mathbf{H}}=0$, which leads to a contradiction. If in this case $\Phi$ has a dense image, then $W_{\Phi,\omega}$ is an injection (see \cite[Proposition 2.6]{erz}).

(iii): We only need to show sufficiency. Assume that $W_{\Phi,\omega}^{*}$ is bounded from below. Then it follows from part (ii) that $W_{\Phi,\omega}$ is an injection with a dense image. However, since $W_{\Phi,\omega}^{*}$ is bounded from below it follows that the image of $W_{\Phi,\omega}$ is closed (see \cite[Exercise 2.49]{fhhmz} with the solution therein). Hence, $W_{\Phi,\omega}$ is a linear homeomorphism, and so $W_{\Phi,\omega}^{*}$ is also a linear homeomorphism (see the same reference). If in this case $W_{\Phi,\omega}^{*}$ is an isometry, then it is a unitary, and so $W_{\Phi,\omega}$ is also a unitary.
\end{proof}

\section{Unitary MO's}

In this section we investigate our main question. Namely, we look for conditions on a NSCF that would prevent it from admitting unitary MO's other than the scalar multiples of the identity. Let us first consider some examples of such conditions.

\begin{example}
Assume that $X$ is a domain in $\C^{n}$ and $\mathbf{F}\ne\left\{0\right\}$ is a NSCF over $X$ that consists of holomorphic functions on $X$. Let $\omega:X\to\C$ be such that $M_{\omega}$ is unitary on $\mathbf{F}$. From part (ii) of Proposition \ref{hc} we may assume that $\omega$ is holomorphic on $X$. Since $M_{\omega}$ is unitary, $M_{\omega}^{*}$ is an isometry on $\mathbf{F}^{*}$, and so from Proposition \ref{rec} it follows that $\left|\omega\left(x\right)\right|=1$ for every $x\in X$ such that $x_{\mathbf{F}}\ne 0_{\mathbf{F}^{*}}$. Let $f\in \mathbf{F}\backslash\left\{0\right\}$. Then for every $x\not\in f^{-1}\left(0\right)$ we have that $x_{\mathbf{F}}\ne 0_{\mathbf{F}^{*}}$, and so $\left|\omega\left(x\right)\right|=1$. Hence, $\omega$ is holomorphic on $X$  and such that $\left|\omega\right|\equiv 1$ on a nonempty open set $X\backslash f^{-1}\left(0\right)$. From the Open Mapping theorem (see \cite[Conclusion 1.2.12]{scheidemann}) it follows that $\omega$ is a constant function.\qed
\end{example}

\begin{remark}
If we dealt with real-valued functions, then $\omega$ would be real-valued. Hence, if $\mathbf{F}$ was a $1$-independent ``real-valued'' NSCF over a connected space $X$, then from part (i) of Proposition \ref{hc}, $\omega$ would be a continuous function on a connected space with valued $\pm 1$. Thus, either $\omega\equiv 1$, or $\omega\equiv -1$.\qed
\end{remark}

\begin{example}\label{ipn}
Let us show that if $\mathbf{F}$ is a $1$-independent NSCF over a connected space $X$, and moreover $\mathbf{F}$ is a Hilbert space, then any unitary MO on $\mathbf{F}$ is a scalar multiple of the identity. Let $\omega:X\to\C$ be such that $M_{\omega}$ is unitary on $\mathbf{F}$. Then $M_{\omega}^{*}$ is an isometry on $\mathbf{F}^{*}$, from where $\left|\omega\right|\equiv 1$ and $$\left<x_{\mathbf{F}},y_{\mathbf{F}}\right>=\left<M_{\omega}^{*}x_{\mathbf{F}},M_{\omega}^{*}y_{\mathbf{F}}\right>=
\left<\omega\left(x\right)x_{\mathbf{F}},\omega\left(y\right)y_{\mathbf{F}}\right>=\omega\left(x\right)\overline{\omega\left(y\right)}\left<x_{\mathbf{F}},y_{\mathbf{F}}\right>.$$ If additionally $\left<x_{\mathbf{F}},y_{\mathbf{F}}\right>\ne 0$, then $\omega\left(x\right)\overline{\omega\left(y\right)}=1=\omega\left(y\right)\overline{\omega\left(y\right)}$, and so $\omega\left(x\right)=\omega\left(y\right)$.

From part (i) of Theorem \ref{barrell} and reflexivity of $\mathbf{F}$ it follows that $\kappa_{\mathbf{F}}$ is a weakly continuous map from $X$ into $\mathbf{F}^{*}$. Let $x\in X$. Since $0<\left\|x_{\mathbf{F}}\right\|^{2}=\left<x_{\mathbf{F}},x_{\mathbf{F}}\right>$, there is an open neighborhood $U$ of $x$ such that $\left<x_{\mathbf{F}},y_{\mathbf{F}}\right>\ne 0$ for every $y\in U$. Hence, $\omega\left(x\right)=\omega\left(y\right)$, and so $\omega$ is a constant on $U$. Since $x$ and $U$ were chosen arbitrarily we get that $\omega$ is locally a constant, and since $X$ is connected, we conclude that $\omega$ is a constant function.\qed
\end{example}

The examples above suggest that the connectedness of $X$ is a natural restriction in the context of our investigation. Indeed, it is easy to construct counterexamples for disconnected spaces. Namely, let $\mathbf{F}$ and $\mathbf{E}$ be $1$-independent NSCF's over topological spaces $X$ and $Y$. Let $Z$ be the disjoint sum of $X$ and $Y$ and let $\mathbf{H}=\left\{h:Z\to \C,~\left.h\right|_{X}\in\mathbf{F},~\left.h\right|_{Y}\in\mathbf{E}\right\}$ endowed with a norm $\|h\|=\sqrt{\|\left.h\right|_{X}\|^{2}+\|\left.h\right|_{Y}\|^{2}}$. It is easy to see that $\mathbf{H}$ is a $1$-independent NSCF over $Z$ and a nonconstant function $\omega=\mathds{1}_{X}-\mathds{1}_{Y}$ gives rise to a unitary MO on $\mathbf{H}$.

On the other hand, there are naturally occurring NSCF's on connected spaces which admit nontrivial unitary MO's. Indeed, for any topological space $X$ the operator $M_{\omega}$ is unitary on the NSCF $\Co_{\8}\left(X\right)$, for any $\omega\in\Co\left(X,\T\right)$.\medskip

Let us analyse Example \ref{ipn}. The proof of the rigidity in that example relies on two ingredients: the different eigenspaces of an isometry are orthogonal and the point evaluations of two points which are ``close'' cannot be orthogonal. It turns out that there is a concept of orthogonality in the general normed spaces that can be utilized to the same effect.

Let $E$ be a normed space. A vector $e\in E$ is called \emph{Birkhoff (or Birkhoff-James) orthogonal} to $f\in E$, if $\|e\|\le \|e+tf\|$ for any $t\in \C$, i.e. $\|e\|=\|Pe\|$, where $P$ is the quotient map from $E$ onto $E\slash\spa\left\{f\right\}$. If $E$ is a Hilbert space, then $P$ is the orthogonal projection onto $E\ominus \spa f$, and so the notion of Birkhoff orthogonality coincides with the usual one. Note however, that in general the Birkhoff orthogonality is NOT a symmetric relation, which is one of the crucial differences between these concepts. This inspired our notation $e\vdash f$ for ``$e$ is Birkhoff orthogonal to $f$''. There are other generalizations of the notion of orthogonality, some of which are symmetric, but we will only use the Birkhoff orthogonality. More details on the subject can be found e.g. in \cite{amw} or \cite[Section 1.4]{fj}. The following lemma shows that different eigenspaces of an isometry on a normed space are Birkhoff orthogonal.

\begin{lemma}\label{isoort}
Let $E$ be a normed space and let $T:E\to E$ be an isometry. If $e,f\in E\backslash\left\{0_{E}\right\}$ are such that $Te=\alpha e$ and $Tf=\beta f$, for some distinct $\alpha,\beta\in\C$, then $e\vdash f$ and $f\vdash e$.
\end{lemma}
\begin{proof}
Let $\gamma=\frac{\beta}{\alpha}\ne 1$, and so $Tf=\gamma\alpha f$. Since $T$ is an isometry, it follows that $\alpha,\beta,\gamma\in\T$, and also $$\|f+\gamma te\|=\|Tf+\gamma tTe\|=\|\gamma\alpha f + \gamma\alpha t e\|=\|f+ te\| ,$$ for any $t\in\C$. Applying this equality $n$ times we get that $\|f+\gamma^{n} te\|=\|f+te\|$, for any $n\in\N$. Since $\gamma\ne 1$ the set $\left\{\gamma^{n}t, n\in\N\right\}$ is either a regular polygon centered at $0$, or a dense subset of $t\T$, and so its convex hull contains $0$. Hence, from the convexity of the function $t\to\|f+te\|$, we get that $\|f\|\le\|f+te\|$, for any $t\in\C$, i.e. $f\vdash e$. Due to symmetry, $e\vdash f$.
\end{proof}

Let $\mathbf{F}$ be a $1$-independent NSCF over a Hausdorff space $X$. Let us introduce a graph structure generated by $\mathbf{F}$. The \emph{Birkhoff graph} of $\mathbf{F}$ is the graph with $X$ serving as a set of vertices, and $x,y\in X$ are joined with an edge if either $x_{\mathbf{F}}\not\vdash y_{\mathbf{F}}$, or $y_{\mathbf{F}}\not\vdash x_{\mathbf{F}}$. The connected components of $Y\subset X$ in this graph are the classes of the minimal equivalence relation on $Y$ which includes all pairs $\left(x,y\right)\in Y\times Y$ such that $x_{\mathbf{F}}\not\vdash y_{\mathbf{F}}$. Now we can state the criterion of the rigidity in terms of the Birkhoff graph.

\begin{proposition}\label{birkgraph0}
Let $\mathbf{F}$ be a $1$-independent NSCF over a Hausdorff space $X$. Let $Y\subset X$ be connected in the Birkhoff graph of $\mathbf{F}$. Let $T:\spa~\kappa_{\mathbf{F}}\left(Y\right)\to \spa~\kappa_{\mathbf{F}}\left(Y\right)$ be a linear isometry such that $y_{\mathbf{F}}$ is an eigenvector of $T$, for every $y\in Y$. Then $T$ is a scalar multiple of the identity.
\end{proposition}
\begin{proof}
Define $\omega:Y\to\C$ by $Ty_{\mathbf{F}}=\omega\left(y\right)y_{\mathbf{F}}$, for $y\in Y$. Let ``$\sim$'' be a relation on $Y$ defined by $x\sim y$ if $\omega\left(x\right)=\omega\left(y\right)$. It is clear that this is an equivalence relation. It follows from Lemma \ref{isoort} that $x_{\mathbf{F}}\not\vdash y_{\mathbf{F}}\Rightarrow x\sim y$. Hence, $\sim$ is an equivalence relation that contains all pairs $\left(x,y\right)\in Y\times Y$ such that $x_{\mathbf{F}}\not\vdash y_{\mathbf{F}}$, and so its classes of equivalence should contain the connected components of $Y$ in the Birkhoff graph of $\mathbf{F}$. Since $Y$ is connected in that graph, it follows that $\omega\left(x\right)=\omega\left(y\right)$, for every $x,y\in Y$. Thus, $\omega\equiv\lambda$, for some $\lambda \in\T$, and so $T=\lambda Id_{\spa~\kappa_{\mathbf{F}}\left(Y\right)}$.
\end{proof}

\begin{corollary}\label{birkgraph}
Let $\mathbf{F}$ be a $1$-independent NSCF over a Hausdorff space $X$ such that the Birkhoff graph of $\mathbf{F}$ is connected. If $\omega:Y\to\C$ is such that $M_{\omega}$ is a unitary on $\mathbf{F}$, then $\omega\equiv\lambda$, where $\lambda \in\T$.
\end{corollary}

In the light of the corollary above we have to find sufficient conditions for a NSCF to have a connected Birkhoff graph. It is natural to expect that the connectedness of the phase space plays a role. In order to extend the proof from Example \ref{ipn} to the general case we have to find out how far can we push ``nearby points cannot have orthogonal point evaluations'' argument. For this we need some additional information about Birkhoff orthogonality (see more in the next section).

Let $E$ be a normed space. For $e\in E\left\{0_{E}\right\}$ let $e^{\parallel}=\left\{\nu\in \overline{B}_{E^{*}}, \left<e,\nu\right>=\|e\|\right\}$, i.e. $e^{\parallel}=\overline{B}_{E^{*}}\bigcap e^{-1}\left(\|e\|\right)$, where $e$ is viewed as a functional on $E^{*}$. This set is closed and convex, and it is easy to see that it is in fact included in $\partial B_{E^{*}}$. It is well-known that $e\vdash f$ if and only if $e^{\parallel}\cap f^{\bot}\ne\varnothing$, where $f^{\bot}\subset E^{*}$. Indeed, if $P:E\to E\slash \spa ~f$ is a quotient map, then $P^{*}$ is the isometry from $\left(E\slash \spa f\right)^{*}$ into $f^{\bot}$ (see the proof of \cite[Proposition 2.6]{fhhmz}), and so $\left\|Pe\right\|=\sup\limits_{\nu\in f^{\bot}\bigcap \overline{B}_{E^{*}}}\left|\left<e,\nu\right>\right|$. Hence, from the weak* compactness of the balanced set $f^{\bot}\bigcap \overline{B}_{E^{*}}$ and weak* continuity of $e$ it follows that $\left\|Pe\right\|=\|e\|$ if and only if there exists $\nu\in e^{\parallel}\cap f^{\bot}$. Using this information we can state the second ingredient of our main result.

\begin{proposition}\label{birkgraph2}
Let $\mathbf{F}$ be a  weakly compactly embedded $1$-independent NSCF over a Hausdorff space $X$. Let $x\in X$ be such that the set $\left\{f\in \overline{B_{\mathbf{F}}}^{\Co\left(X\right)}\left|f\left(x\right)=\|x_{\mathbf{F}}\|\right.\right\}$ is equicontinuous. Then the closed neighborhood \footnote{Recall that a \emph{neighborhood} of a vertex $x$ in the graph is the set of all vertices joined with $x$, while a \emph{closed neighborhood} of $x$ is the union of the neighborhood of $x$ and $\left\{x\right\}$.} of $x$ in the Birkhoff graph of $\mathbf{F}$ is a neighborhood of $x$ in $X$.
\end{proposition}
\begin{proof}
Let $E=\spa~\kappa_{\mathbf{F}}\left(X\right)\subset \mathbf{F}^{*}$. Since from Theorem \ref{barrell2} the closed unit ball of $E^{*}$ is $\overline{B_{\mathbf{F}}}^{\Co\left(X\right)}$, it follows that $x_{\mathbf{F}}^{\parallel}=\left\{f\in \overline{B_{\mathbf{F}}}^{\Co\left(X\right)}\left|f\left(x\right)=\|x_{\mathbf{F}}\|\right.\right\}$. Since this set is equicontinuous, there is an open neighborhood $U$ of $x$ such that $\left|f\left(y\right)-f\left(x\right)\right|\le \frac{1}{2}\|x_{\mathbf{F}}\|$, and so $f\left(y\right)\ne 0$, for every $y\in U$ and $f\in x_{\mathbf{F}}^{\parallel}$. Hence, $x_{\mathbf{F}}\not\vdash y_{\mathbf{F}}$, for every $y\in U$, and so $U$ is contained in the closed neighborhood of $x$ in the Birkhoff graph of $\mathbf{F}$.
\end{proof}

\begin{corollary}\label{birkgraph3}
Let $\mathbf{F}$ be a weakly compactly embedded $1$-independent NSCF over $X$. Then every connected $Y\subset X$ is connected in the Birkhoff graph of $\mathbf{F}$, if one of the following conditions is satisfied:
\item[(i)] For any $x\in X$ the set $\left\{f\in \overline{B_{\mathbf{F}}}^{\Co\left(X\right)}\left|f\left(x\right)=\|x_{\mathbf{F}}\|\right.\right\}$ is equicontinuous;
\item[(ii)] For any $x\in X$ the set $\left\{f\in \overline{B_{\mathbf{F}}}^{\Co\left(X\right)}\left|f\left(x\right)=\|x_{\mathbf{F}}\|\right.\right\}$ is finitely dimensional;
\item[(iii)] $X$ is compactly generated, and for any $x\in X$ the set $\left\{f\in \overline{B_{\mathbf{F}}}^{\Co\left(X\right)}\left|f\left(x\right)=\|x_{\mathbf{F}}\|\right.\right\}$ is compact.
\end{corollary}
\begin{proof}
If (i) is satisfied, then the components of the Birkhoff graph of $\mathbf{F}$ are disjoint and open, due to Proposition \ref{birkgraph2}. Hence, every connected subset of $X$ is completely included in one of these components, and so is graph-connected.

At the same time, (iii) implies (i) by virtue of Arzela-Ascoli theorem. Moreover, (ii) also implies (i) since every bounded finitely dimensional set is always equicontinuous. Indeed, such set is contained in a convex hull of a finite set. Since a finite set of functions is always equicontinuous, and a convex hull of an equicontinuous set is equicontinuous, the implication follows.
\end{proof}

Thus, if the conditions of the corollary above are fulfilled and $X$ is connected, the only unitary MO's on $\mathbf{F}$ are the scalar multiples of the identity, by virtue of Corollary \ref{birkgraph}. However, these conditions can be difficult to check, and so it is desirable to find stronger conditions which are more readily verifiable. It turns out that one such condition is of geometric nature. A normed space $F$ is called \emph{nearly strictly convex} if the convex subsets of the unit sphere $\partial B_{F}$ are precompact (i.e. totally bounded) in $F$. Note that if $F$ is a Banach space, then it is nearly strictly convex if and only if the closed convex subsets of $\partial B_{F}$ are compact. It is clear that finitely dimensional normed spaces are nearly strictly convex, as well as strictly- and uniformly convex normed spaces, including Hilbert spaces and taking $L^{p}$ spaces, for $p\in\left(1,+\8\right)$ (see \cite[Definition 7.6, Definition 9.1 and Theorem 9.3]{fhhmz}). Furthermore, a linear subspace of a nearly strictly convex normed space is nearly strictly convex. Also, this class of normed spaces is closed under $l^{p}$ sums, for $p\in\left(1,+\8\right)$ (see \cite{ns} and also Remark \ref{finco}). We can now state our main results.

\begin{theorem}\label{main}
Let $\mathbf{F}$ be a $1$-independent NSCF over a connected compactly generated space $X$. If $\omega:Y\to\C$ is such that $M_{\omega}$ is unitary on $\mathbf{F}$ then $\omega\equiv\lambda$, for some $\lambda\in\T$, provided that one of the following conditions is satisfied:
\item[(i)] $\mathbf{F}$ is compactly embedded;
\item[(ii)] $\mathbf{F}$ is weakly compactly embedded and nearly strictly convex, and $\overline{B}_{\mathbf{F}}$ is closed in $\Co\left(X\right)$;
\item[(iii)] $\mathbf{F}$ is weakly compactly embedded and $\mathbf{F}^{**}$ is nearly strictly convex;
\item[(iv)] $\mathbf{F}$ is reflexive and nearly strictly convex.
\end{theorem}
\begin{proof}
Let $x\in X$ be arbitrary. In the light of Corollary \ref{birkgraph} and the condition (iii) of Corollary \ref{birkgraph3} it is enough to show that each of the conditions (i)-(iv) imply that the set $L_{x}=\left\{f\in \overline{B_{\mathbf{F}}}^{\Co\left(X\right)}\left|f\left(x\right)=\|x_{\mathbf{F}}\|\right.\right\}$ is compact in $\Co\left(X\right)$.

If (i) holds, then $\overline{B_{\mathbf{F}}}^{\Co\left(X\right)}$ is compact, and so is its closed subset $L_{x}$.

Assume that (ii) holds. Then $\overline{B}_{\mathbf{F}}=\overline{B_{\mathbf{F}}}^{\Co\left(X\right)}$, and so $L_{x}$ is a closed convex subset of $\partial B_{\mathbf{F}}$. Since $\mathbf{F}$ is nearly strictly convex it follows that $L_{x}$ is compact in $\mathbf{F}$. Since the topology of $\mathbf{F}$ is stronger than the compact-open topology, we conclude that $L_{x}$ is compact in $\Co\left(X\right)$.

If (iii) holds, then since $\overline{B_{\mathbf{F}}}^{\Co\left(X\right)}=J_{\mathbf{F}}^{**}\overline{B}_{\mathbf{F}^{**}}$, we have that $L_{x}$ is the image under $J_{\mathbf{F}}^{**}$ of the set $N_{x}=\left\{g\in \overline{B}_{\mathbf{F}^{**}}\left|\left<g,x_{\mathbf{F}}\right>=\|x_{\mathbf{F}}\|\right.\right\}$. Clearly, $N_{x}\subset \partial B_{\mathbf{F}^{**}}$ and is a convex set. Since $\mathbf{F}^{**}$ is nearly strictly convex, it follows that $N_{x}$ is compact. Hence, as $J_{\mathbf{F}}^{**}$ is continuous from $\mathbf{F}^{**}$ into $\Co\left(X\right)$, it follows that $L_{x}$ is also compact.

Finally, observe that (iv) implies (iii). Indeed, every reflexive NSCF is weakly compactly embedded, and if $\mathbf{F}$ is reflexive and nearly strictly convex, then $\mathbf{F}^{**}=\mathbf{F}$ is nearly strictly convex.
\end{proof}

\begin{remark}
Note that the condition (iv) is only imposed on the Banach space properties of $\mathbf{F}$ and has nothing to do with its embedding into $\Co\left(X\right)$.
\qed\end{remark}
\begin{remark}
It is desirable to relax the conditions of the theorem. In fact, at the moment we do not have an example of a non-trivial unitary MO on an either weakly compactly embedded or nearly strictly convex NSCF over a connected compactly generated space.
\qed\end{remark}

The statement can be adjusted to get rid of the $1$-independence.

\begin{proposition}\label{main2}
Let $\mathbf{F}$ be a NSCF over a Hausdorff space $X$ such that the set $\left\{x\in X\left|x_{\mathbf{F}}\ne0_{\mathbf{F}^{*}}\right.\right\}$ is connected and one of the conditions of Theorem \ref{main} are met. Then every unitary MO on $\mathbf{F}$ is a scalar multiple of $Id_{\mathbf{F}}$.
\end{proposition}

Let us consider an example of a NSCF over a disconnected space whose Birkhoff graph is connected nonetheless.

\begin{example}
Let $\left(X,d\right)$ and $z\in X$ be as in Example \ref{lip}. Additionally assume that the distances between components of $X$ is less than $1$. We will show that the Birkhoff graph of $\mathbf{F}=Lip\left(X,d\right)$ is connected. Let $x\in X$. Since the distance between components containing $x$ and $z$ is less than $1$, there are $y$ in the component of $x$ and $w$ in the component of $z$ such that $d\left(y,w\right)<1$. Then $\|w_{\mathbf{F}}\|=\max\left\{1,d\left(w,z\right)\right\}> d\left(w,y\right)= \|w_{\mathbf{F}}+\left(-1\right)y_{\mathbf{F}}\|$, and so $w_{\mathbf{F}}\not\vdash y_{\mathbf{F}}$. Due to Corollary \ref{birkgraph3} there are paths from $x$ to $y$ and from $w$ to $z$ in the Birkhoff graph, while $y$ and $w$ are joined with an edge. Hence, there is a path from $x$ to $z$, and since $x$ was chosen arbitrarily, we conclude that the Birkhoff graph of $\mathbf{F}$ is connected. Thus, due to Proposition \ref{birkgraph}, the only unitary MO's on $\mathbf{F}$ are the scalar multiples of the identity.
\qed\end{example}

Similarly to Theorem \ref{main}, we can prove an analogous statement for WCO's.

\begin{proposition}\label{mainw}
Let $\mathbf{F}$ be a $1$-independent NSCF over a Hausdorff space $X$ that satisfies one of the conditions of Theorem \ref{main}. Let $\mathbf{E}$ be a NSCF over a Hausdorff space $Y$. If $\Phi:Y\to X$ is such that $\Phi\left(Y\right)$ is connected, and $\omega,\upsilon:Y\to\C\backslash\left\{0\right\}$ are such that there is a unitary $S:\mathbf{F}\to \mathbf{F}$ such that $W_{\Phi,\omega}=W_{\varphi,\upsilon}S$ (e.g. if both $W_{\Phi,\omega}$ and $W_{\Phi,\upsilon}$ are unitaries), then $\upsilon=\lambda \omega$, for some $\lambda\in\T$.
\end{proposition}
\begin{proof}
First, note that $S^{*}$ is an isometry such that $S^{*}\Phi\left(y\right)_{\mathbf{F}}=\frac{\omega\left(y\right)}{\upsilon\left(y\right)}\Phi\left(y\right)_{\mathbf{F}}$, for every $y\in Y$. Since $\Phi\left(Y\right)$ is connected, the result is obtained by combining Proposition \ref{birkgraph0} with Corollary \ref{birkgraph3}.
\end{proof}

\begin{remark}
It is clear that $\Phi\left(Y\right)$ is connected in the case when $Y$ is connected and $\Phi$ is continuous, and also in the case when $X$ is connected and $\Phi$ is a surjection. Moreover, continuity of $\Phi$ often holds automatically for WCO's between NSCF's (see \cite[Corollary 3.3, Theorem 3.10 and Theorem 3.12]{erz}), while surjectivity of $\Phi$ also can be deduced from the properties of the WCO (see in \cite[Proposition 2.11]{erz}). In fact, if $X$ is a manifold, $\mathbf{F}$ is $2$-independent with $x\to \left\|x_{\mathbf{F}}\right\|$ continuous, $\lim\limits_{\8}\left\|x_{\mathbf{F}}\right\|=+\8$ and bounded functions form a dense subset of $\mathbf{F}$, then $\mathbf{F}$ is rigid in the following stronger sense: if $\Phi:X\to X$ and $\omega,\upsilon:X\to\C$ are such that $W_{\Phi,\omega}$ and $W_{\Phi,\upsilon}$ are unitaries, then $\Phi$ is a self-homeomorphism of $X$ and $\omega=\lambda \upsilon$, for some $\lambda\in\T$, are continuous and non-vanishing.
\qed\end{remark}

Up to this point in this section the word ``unitary'' could be replaced with the word ``co-isometry''. \footnote{An operator between normed spaces is called a \emph{co-isometry} if its adjoint is an isometry.} Note however, that due to part (iii) of Corollary \ref{winj2}, any MO between complete NSCF's, which is a co-isometry is automatically a unitary. Let us conclude the section with a version of Theorem \ref{main} for non-surjective isometries.

\begin{proposition}
Let $\mathbf{F}$ be a $1$-independent NSCF over a Hausdorff space $X$ such that one of the conditions of Theorem \ref{main} are met. Let $\omega:X\to\C\backslash\left\{0\right\}$ be such that $M_{\omega}$ is an isometry on $\mathbf{F}$. Assume that $Y$ is a dense connected subset of $X$ such that for every $x\in Y$ there is $f\in \mathbf{F}$ such that $f\left(x\right)\ne 0$ and $\omega^{-n}f\in \mathbf{F}$, for every $n\in\N$. Then $\omega\equiv\lambda$, for some $\lambda\in\T$.
\end{proposition}
\begin{proof}
First, note that $\omega$ is continuous by virtue of part (i) of Proposition \ref{hc}, and since it does not vanish, $\frac{1}{\omega}$ is also continuous.

Let $\mathbf{E}=\bigcap\limits_{n\in\N}M_{\omega}^{n}\mathbf{F}$, which is a closed subspace of $\mathbf{F}$. Then $\mathbf{E}$ is a NSCF over $X$ that satisfies conditions of Theorem \ref{main}. Indeed, a closed subspace of a (weakly) compactly embedded NSCF is also a a (weakly) compactly embedded NSCF; a closed subspace of nearly strictly convex or a reflexive normed space is nearly strictly convex or reflexive; if $\mathbf{F}^{**}$ is nearly strictly convex, then so is $\mathbf{E}^{**}\subset\mathbf{F}^{**}$; finally, if $\overline{B}_{\mathbf{F}}$ is closed in $\Co\left(X\right)$, then so is $\overline{B}_{\mathbf{E}}=\bigcap\limits_{n\in\N}M_{\omega}^{n}\overline{B}_{\mathbf{F}}$, as $M_{\omega}$ is a self-homeomorphism of $\Co\left(X\right)$.

The set $Z=\left\{x\in X\left|x_{\mathbf{E}}\ne0_{\mathbf{E}^{*}}\right.\right\}$ contains $Y$. Indeed, for every $x\in Y$ there is $f\in \mathbf{F}$ such that $f\left(x\right)\ne 0$ and $f\in M_{\omega}^{n}\mathbf{F}$, for every $n\in\N$. Hence, $Z$ is connected and dense in $X$. Since $M_{\omega}$ is a unitary on $\mathbf{E}$, by Proposition \ref{main2}, it follows that $\omega$ is a constant function on $Z$. As $Z$ is dense in $X$ and $\omega$ is continuous, the result follows.
\end{proof}

\section{More on geometry of normed spaces}\label{boo}

In this section we gather some leftover results and remarks that are not directly related to NSCF's, and instead are given in the context of abstract normed spaces. Let us start by revisiting one of intuitive aspects of the orthogonality in the inner product spaces. Namely, one can view orthogonal vectors as ``separated''. More precisely, for any $e\ne 0_{E}$ in a Hilbert space $E$, $e^{\bot}$ is a hyperplane, which is a closed convex (and so weakly closed) set not containing $e$. It is natural to ask whether the same phenomenon holds in general normed spaces.

As was already mentioned, the relation $\vdash$ of Birkhoff orthogonality is not symmetric in general normed spaces. Hence, if $E$ is a normed space and $e\ne 0_{E}$, we can consider distinct orthogonal complements $e^{\vdash}=\left\{f\in E\left|e\vdash f\right.\right\}$ and $^{\vdash}e=\left\{f\in E\left|f\vdash e\right.\right\}$. From the characterization of Birkhoff orthogonality, $^{\vdash}e$ is the set of all maximal elements of functionals in $e^{\bot}\subset E^{*}$, while $e^{\vdash}=\bigcup\limits_{\nu\in e^{\parallel}}\nu^{\bot}$. It is easy to see that the set $\left\{\left(e,f\right)\in E\times E\left|e\vdash f\right.\right\}$ is norm-closed in $E\times E$, and so both $e^{\vdash}$ and $^{\vdash}e$ are closed with respect to the norm topology on $E$. However we cannot immediately conclude that these sets are weakly closed since they are usually not convex. More specifically, $e^{\vdash}$ is a union of hyperplanes. It turns out that the key factor in the question of when this set is weakly convex is how ``many'' hyperplanes are involved.

\begin{theorem}\label{bort}
For a nonzero vector $e\in E$ the following are equivalent:
\item[(i)] $e^{\vdash}$ is weakly closed;
\item[(ii)] $e$ is weakly separated from $e^{\vdash}$ (i.e. $e$ does not belongs to the weak closure of $e^{\vdash}$);
\item[(iii)] $e^{\vdash}$ is not weakly dense in $E$;
\item[(iv)] The set $e^{\parallel}$ is of finite-dimension.
\end{theorem}
\begin{proof}
(i)$\Rightarrow$(ii)$\Rightarrow$(iii) is trivial. Let us prove (iii)$\Rightarrow$(iv): Assume, there are $g\in E$ and $\left\{\nu_{1},..., \nu_{n}\right\}\subset E^{*}$, such that $V=\left\{f\in E \left|\forall j\in\overline{1,n}~\left|\left<f-g,\nu_{j}\right>\right|<1\right.\right\}$ does not intersect $e^{\vdash}$. Take a nonzero $f\in \left\{\nu_{1},..., \nu_{n}\right\}^{\bot}$. Then $\left<g+tf-g,\nu_{j}\right>=0$, for all $j\in\overline{1,n}$, and so $g+tf\in V$, for any $t\in \C$. For any $\nu\in e^{\parallel}$ we have that $\left<g+tf,\nu\right>=\left<g,\nu\right>+t\left<f,\nu\right>$. If $\left<f,\nu\right>\ne 0$, for $t=-\frac{\left<g,\nu\right>}{\left<f,\nu\right>}$ we have that $\left<g+tf,\nu\right>=0$, which contradicts the assumption $V\bigcap e^{\vdash}=\varnothing$. Hence $\nu\in f^{\bot}$, and from the arbitrariness of $f$ and $\nu$, we get that $e^{\parallel}\subset \left(\left\{\nu_{1},..., \nu_{n}\right\}^{\bot}\right)^{\bot}=\spa \left\{\nu_{1},..., \nu_{n}\right\}$.\medskip

(iv)$\Rightarrow$(i): Assume that $e^{\parallel}$ is finite-dimensional. Since this set is bounded, there is a finite collection $D=\left\{\nu_{1}, \nu_{2},..., \nu_{n}\right\}\subset E^{*}$, such that $e^{\parallel}\subset \conv D$. Let $g\not\in e^{\vdash}$. Then $\left<g,\nu\right>\ne 0$, for any $\nu\in e^{\parallel}$, and due to weak* compactness of $D$ and continuity of $g$ as a functional on $D$, there is $\delta>0$ such that $\left|\left<g,\nu\right>\right|\ge \delta$, for any $\nu\in e^{\parallel}$. The set $U=\left\{f\in E\left|\forall j\in\overline{1,n}~\left|\left<f-g,\nu_{j}\right>\right|<\delta\right.\right\}$ is a weakly open neighborhood of $g$, which is disjoint from $e^{\vdash}$. Indeed, for any $\nu\in e^{\parallel}$ there are $t_{1},...,t_{n}$, such that $\sum\limits_{j=1}^{n}t_{j}=1$ and $\nu=\sum\limits_{j=1}^{n}t_{j}\nu_{j}$. Then for any $f\in U$ we have that
\begin{align*}
\left|\left<f,\nu\right>\right|&=\left|\left<g,\nu\right>+\left<f-g,\nu\right>\right|=\left|\left<g,\nu\right>+\sum\limits_{j=1}^{n}t_{j}\left<f-g,\nu_{j}\right>\right|\\
&\ge \left|\left<g,\nu\right>\right|-\sum\limits_{j=1}^{n}t_{j}\left|\left<f-g,\nu_{j}\right>\right|>\delta-\sum\limits_{j=1}^{n}t_{j}\delta=0,
\end{align*}
and so $f\not\in e^{\vdash}$. Thus, $g$ is weakly separated from $e^{\vdash}$, and since $g$ was chosen arbitrarily we conclude that $e^{\vdash}$ is weakly closed.
\end{proof}

\begin{remark}
It would also be interesting to find necessary and sufficient conditions for weak closeness $^{\vdash}e$.\medskip
\end{remark}

Let us now state an interpretation of Theorem \ref{main} in the context of abstract normed spaces.

\begin{theorem}
Let $E$ be a normed space and let $T:E\to E$ be an isometry. Let $D\subset E\backslash\left\{0_{E}\right\}$ consist of eigenvectors of $T$ such that $\overline{\spa~ D}=E$. Then $T=\lambda Id_{E}$, for some $\lambda\in\T$ provided that one of the following conditions is satisfied:
\item[(i)] $D$ is connected in the norm topology;
\item[(ii)] $D$ is weakly connected and for each $e\in D$ the set $\left\{\nu\in \overline{B}_{E^{*}}, \left<e,\nu\right>=\|e\|\right\}$ is of finite dimension;
\item[(iii)] $D$ is bounded and weakly connected and $E^{*}$ is separable and nearly strictly convex.
\end{theorem}
\begin{proof}
We will only show the sufficiency of (iii). The sufficiency of (i) and (ii) is shown similarly. We can view elements of $E^{*}$ as a NSCF over $\left(D,weak\right)$. Since $\overline{\spa~ \kappa_{E^{*}}}=E$, it follows from Corollary \ref{barrell3} that $E^{*}$ is a weakly compactly embedded, and also that $\overline{B}_{E^{*}}$ is closed in $\Co\left(D\right)$.

Since $E^{*}$ is separable it follows that a bounded set $D$ is weakly metrizable (see \cite[Proposition 3.106]{fhhmz}). Hence, $E^{*}$ is a nearly strictly convex NSCF over a connected metrizable space $D$. Thus, $E^{*}$ satisfies the condition (i) of Theorem \ref{main}, and so its Birkhoff graph is connected. By virtue of Proposition \ref{birkgraph0} we conclude that $T$ is a constant multiple of the identity.
\end{proof}

Let us conclude the article with discussing nearly strictly convex normed spaces. A lot of facts about strictly convex normed spaces have analogues for the nearly strictly convex case. For example, it is easy to see that if $T$ is a linear map from a nearly strictly convex normed space $E$ into a normed space $F$ such that $T\overline{B}_{E}=\overline{B}_{F}$, then $F$ is also nearly strictly convex. Consequently, if $H$ is a subspace of $E$ which is a reflexive Banach space, then $E\slash H$ is nearly strictly convex (for the proof of the fact that the quotient map maps $\overline{B}_{E}$ onto $\overline{B}_{E\slash H}$ see the proof of \cite[Theorem 2.2.5]{istr}). Note that reflexivity of $H$ is essential since any Banach space can be obtained as a quotient of a strictly convex space (see \cite[Theorem 2.2.7]{istr}). Now let us discuss when the sum of nearly strictly convex normed spaces is nearly strictly convex. We start with a finite sum (we omit the proof in favour of the infinite case).

\begin{proposition}
Let $\rho$ be a strictly convex norm on $\R^{n}$ which is invariant with respect to the the reflection over the coordinate hyperplanes. Let $E_{1},...,E_{n}$ be nearly strictly convex normed spaces. Then $E_{1}\times ...\times E_{n}$ is nearly strictly convex with respect to the norm $\left\|\left(e_{1},...,e_{n}\right)\right\|=\rho\left(\|e_{1}\|_{E_{1}},...,\|e_{n}\|_{E_{n}}\right)$, $\left(e_{1},...,e_{n}\right)\in E_{1}\times ...\times E_{n}$.
\end{proposition}

The analogous statement for the case of the infinite sum is more involved.

\begin{proposition}\label{finco}
Let $\rho:\left[0,+\8\right)^{\N}\to \left[0,+\8\right]$ be a functional that satisfies the following conditions:
\begin{itemize}
\item $\rho\left(0_{\R^{\N}}\right)=0$; $\rho\left(\left\{0,...,0,1,0,...\right\}\right)<+\8$;
\item Positive homogeneity: $\rho\left(\lambda u\right)=\lambda\rho\left(u\right)$, for any $u\in \left[0,+\8\right)^{\N}$ and $\lambda>0$;
\item Strict subadditivity: $\rho\left(u+v\right)\le\rho\left(u\right)+\rho\left(v\right)$, for any $u,v\in \left[0,+\8\right)^{\N}$; if $\rho\left(u+v\right)=\rho\left(u\right)+\rho\left(v\right)$ then either $v=\lambda u$, for some $\lambda\ge 0$, or $u=0_{\R^{\N}}$;
\item Monotonicity: $\rho\left(u+v\right)\ge\rho\left(u\right)$, for any $u,v\in \left[0,+\8\right)^{\N}$;
\item Absolute continuity: if $\rho\left(\left\{u_{n}\right\}_{n\in\N}\right)<+\8$, for some $\left\{u_{n}\right\}_{n\in\N}\in \left[0,+\8\right)^{\N}$, then $\rho\left(\left\{0,...,0,u_{n},u_{n+1},...\right\}\right)\to 0$, $n\to\8$.
\end{itemize}
Let $\left\{E_{n}\right\}_{n\in\N}$ be a sequence of nearly strictly convex normed spaces. Define $\|\cdot\|:\prod\limits_{n\in\N}E_{n}\to \left[0,+\8\right]$ by $\left\|\left\{e_{n}\right\}_{n\in\N}\right\|=\rho\left(\left\{\|e_{n}\|_{E_{n}}\right\}_{n\in\N}\right)$. Then $E=\left\{e\in\prod\limits_{n\in\N}E_{n},~ \|e\|<+\8 \right\}$ with the norm $\|\cdot\|$ is a nearly strictly convex normed space.
\end{proposition}
\begin{proof}
For $n\in\N$ let $\overline{E_{n}}$ be the completion of $E_{n}$. Let $\left(\tilde{E},\|\cdot\|\right)$ be a normed space, constructed from $\left\{\overline{E_{n}}\right\}_{n\in\N}$ analogously to construction of $E$. We leave it to the reader to verify that $E$ and $\tilde{E}$ are linear spaces, $\|\cdot\|$ is a norm on $\tilde{E}$, and $E$ is a subspace of $\tilde{E}$. Let us prove that $E$ is nearly strictly convex.

First, using arguments similar to the proof of \cite[Theorem 2.2.1]{istr}, one can show that if $\varnothing\ne D\subset\partial B_{E}$ is convex, and $D_{n}$ is the image of $D$ under the natural projection from $E$ onto $E_{n}$, then $D_{n}$ is a convex subset of a sphere in $E_{n}$. Let $r_{n}$ be the radius of that sphere. For any $e\in D$ we have that $\|e\|=\rho\left(\left\{r_{n}\right\}_{n\in\N}\right)=1$, and so for any $f\in \prod\limits_{n\in\N}r_{n}\partial B_{\overline{E_{n}}}$ we get $\|f\|=\rho\left(\left\{r_{n}\right\}_{n\in\N}\right)=1$. Hence, $B=\prod\limits_{n\in\N}r_{n}\partial B_{\overline{E_{n}}}\subset\partial B_{\tilde{E}}$. Let us show that the norm topology on $B$ is weaker than the product topology. Let $e=\left\{e_{n}\right\}_{n\in\N}\in B$ and let $\varepsilon>0$. Since $\rho\left(\left\{0,0,...,0,r_{n},r_{n+1},...\right\}\right)\to 0$, $n\to\8$, there is $m\in\N$ such that $\rho\left(\left\{0,0,...,0,r_{m},r_{m+1},...\right\}\right)<\frac{\varepsilon}{3}$. Let $c_{n}=\rho\left(\left\{0,...,0,1,0,...\right\}\right)<+\8$, where the $1$ is on the $n$-th position. Then, for $f=\left\{f_{n}\right\}_{n\in\N}\in B$ such that $\|e_{n}-f_{n}\|_{\overline{E_{n}}}<\frac{\varepsilon}{3m\max\left\{1,c_{1},...,c_{m}\right\}}$, for every $n\in\overline{1,m}$, we have
\begin{align*}
\|e-f\|&\le \sum\limits_{n=1}^{m}c_{n}\|e_{n}-f_{n}\|_{\overline{E_{n}}}+\|\left\{0,0,...,0,e_{m},e_{m+1},...\right\}\|+\|\left\{0,0,...,0,f_{m},f_{m+1},...\right\}\|  \\
&< \sum\limits_{n=1}^{m}c_{n}\frac{\varepsilon}{3m\max\left\{1,c_{1},...,c_{m}\right\}}+2\rho\left(\left\{0,0,...,0,r_{m},r_{m+1},...\right\}\right)\le \frac{\varepsilon}{3}+2\frac{\varepsilon}{3}=\varepsilon.
\end{align*}
Since $e$ and $\varepsilon$ were chosen arbitrarily, we conclude that $\|\cdot\|$ induces a topology on $B$ weaker than the product topology. For every $n\in\N$, since $E_{n}$ is nearly strictly convex, it follows that $D_{n}$ is precompact in $E_{n}$. Then the closure $\overline{D_{n}}$ of $D_{n}$ in $\overline{E_{n}}$ is compact. Let $D'= \prod\limits_{n\in\N}\overline{D_{n}}\subset B$, which is a compact set in the product topology, and so is compact in $B$. Since $D\subset D'$ we conclude that $D$ is relatively compact in $\tilde{E}$, and so precompact in $E$, and so $E$ is nearly strictly convex.
\end{proof}

Consider an example of a nearly strictly convex Banach space whose spheres contain infinite dimensional convex sets.

\begin{example}
Let $F=\left.\bigoplus\limits_{n\in\N}\right.^{2} l^{\8}_{2}$, be the $l^{2}$ direct sum of infinite number of copies of $\C^{2}$ with the $l^{\8}$ norm. By virtue of Proposition \ref{finco} this normed space is nearly strictly convex. Let $D_{n}=\left\{\frac{1}{n}\oplus t\left|t\in\left[-\frac{1}{n},\frac{1}{n}\right]\right.\right\}$ be a convex subset of a sphere in $l^{\8}_{2}$ of radius $\frac{1}{n}$.  From the proof of Proposition \ref{finco} it follows that $\prod\limits_{n\in\N}D_{n}$ is an infinite-dimensional convex subset of a sphere in $F$.
\qed\end{example}

Now consider an example of a non-strictly convex Banach space, such that the convex subsets of its unit sphere are at most one-dimensional.

\begin{example}
Let $H$ be a Hilbert space, and let $E=H\oplus_{1} \C$. Assume that $e,f\in H$ and $a,b\in \C$ are such that $\|e\|+\left|a\right|=\|f\|+\left|b\right|=\left\|\frac{e+f}{2}\right\|+\left|\frac{a+b}{2}\right|=1$. Without loss of generality we may assume that $e\ne 0_{H}$. Due to strict convexity of $H$ there are $\alpha,\beta\ge 0$ such that $f=\alpha e$ and $b=\beta a$ (or else $a=0$). Since we also have $\|f\|+\left|b\right|=1$, it follows that the convex subsets of the unit sphere that contain $e\oplus a$ are contained in $\left\{\alpha e\oplus \frac{1-\alpha\|e\|}{\left|a\right|}a\left|\alpha \in\left[0,\frac{1}{\|e\|}\right]\right.\right\}$, when $a\ne 0$, and $\left\{\left(1-\left|\gamma\right|\right) e\oplus \gamma ,~\left|\gamma\right|\le 1\right\}$, when $a=0$.

Note, that the dual $E^{*}=H\oplus_{\8} \C$ satisfies the conditions of Theorem \ref{bort}, and so right Birkhoff orthogonal complements are weakly closed in $E^{*}$.\qed\end{example}

\begin{remark}
It is clear that having a nearly strictly convex subset of finite co-dimension does not imply nearly strictly convexity. Indeed, even if $E$ is a Hilbert space, $E\oplus_{\8}\C$ is not nearly strictly convex. However, one can ask whether it is true that if $E$ is quasi-reflexive (i.e. such that $\dim E^{**}\slash E<+\8$) and nearly strictly convex, then $E^{**}$ is also nearly strictly convex.

Also, it is interesting whether nearly strict convexity of a normed space implies nearly strictly convexity of its completion. Furthermore, one can study a property stronger than nearly strictly convexity: instead of precompactness of closed convex subsets of the unit sphere we can demand compactness. Clearly, the two conditions are equivalent in the event when the normed space is complete.

Finally, one can ask whether it is true that if $E$ is nearly strictly convex, then there is a strictly convex subspace of $E$ of finite codimension.
\qed\end{remark}

\end{document}